\documentclass[11pt]{amsart}

\theoremstyle{plain}
\newtheorem{theorem}                {Theorem}      [section]
\newtheorem*{theorem*}                {Theorem \ref{thm:appl}}
\newtheorem{proposition}  [theorem]  {Proposition}
\newtheorem{corollary}    [theorem]  {Corollary}
\newtheorem{lemma}        [theorem]  {Lemma}

\theoremstyle{definition}

\newtheorem{remark}       [theorem]  {Remark}
\newtheorem{definition}   [theorem]  {Definition}

\DeclareMathOperator{\trace}{trace} 
\DeclareMathOperator{\grad}{grad}

\DeclareMathOperator{\Div}{div} \DeclareMathOperator{\id}{I}

\DeclareMathOperator{\vol}{Vol}

\DeclareMathOperator{\Span}{span}
 
\DeclareMathOperator{\cst}{constant}

\numberwithin{equation}{section}

\begin{document}

\title[CMC biconservative surfaces]
{CMC biconservative surfaces in $\mathbb{S}^n\times\mathbb{R}$ and $\mathbb{H}^n\times\mathbb{R}$}

\author{Dorel~Fetcu}
\author{Cezar~Oniciuc}
\author{Ana~Lucia~Pinheiro}

\address{Department of Mathematics and Informatics\\
Gh. Asachi Technical University of Iasi\\
Bd. Carol I no. 11 \\
700506 Iasi, Romania} \curraddr{Department of Mathematics, Federal University of Bahia, Av.
Adhemar de Barros s/n, 40170--110 Salvador, BA, Brazil} \email{dfetcu@math.tuiasi.ro}

\address{Faculty of Mathematics\\ Al. I. Cuza University of Iasi\\
Bd. Carol I, 11 \\ 700506 Iasi, Romania} \email{oniciucc@uaic.ro}

\address{Department of Mathematics, Federal University of Bahia, Av.
Adhemar de Barros s/n, 40170-110 Salvador, BA, Brazil} \email{anapinhe@ufba.br}

\thanks{The first author was supported by a grant of CNPq, Brazil, BJT 373672/2013--6. The second
author was supported by a grant of the Romanian National Authority
for Scientific Research, CNCS -- UEFISCDI, project number
PN-II-ID-PCE-2012-4-0640. The third author was supported by a grant of CNPq, Brazil.}



\begin{abstract} We classify non-minimal biconservative surfaces with parallel mean curvature vector field in $\mathbb{S}^n\times\mathbb{R}$ and $\mathbb{H}^n\times\mathbb{R}$. When these surfaces do not lie in $\mathbb{S}^n$ or $\mathbb{H}^n$ and they are not vertical cylinders, we find their explicit (local) equation. We also prove a result on the compactness of biconservative surfaces with constant mean curvature in Hadamard manifolds. 
\end{abstract}

\maketitle

\section{Introduction}

Closely related to the theory of {\it biharmonic submanifolds}, the study of {\it biconservative submanifolds} is a very recent and interesting topic in the field of Differential Geometry. In general, a biharmonic map $\psi:(M,g)\rightarrow(\bar M,h)$ between two Riemannian manifolds is a critical point
of the \textit{bienergy functional}
$$
E_2:C^{\infty}(M,\bar M)\rightarrow\mathbb{R},\quad E_{2}(\psi)=\frac{1}{2}\int_{M}|\tau(\psi)|^{2}\ dv,
$$
where $\tau(\psi)$ is the tension field of $\psi$. These critical points are given by the vanishing of the {\it bitension field} $\tau_2(\psi)$ of $\psi$. If $\psi:(M,g)\rightarrow(\bar M,h)$ is a biharmonic Riemannian immersion, then $M$ is called a biharmonic submanifold of $\bar M$.

Now, consider a fixed map $\psi:M\rightarrow(\bar M,h)$ and look at $E_2$ as being defined on the set of all Riemannian metrics on $M$. What we get is a new functional whose critical points, that this time are Riemannian metrics, are given by the vanishing of the {\it stress-energy tensor} $S_2$, that satisfies
$$
\Div S_2=-\langle\tau_2(\psi),d\psi\rangle.
$$ 
A submanifolds that satisfies $\Div S_2=0$ is called a {\it biconservative submanifold} and it is easy to see that a submanifold is biconservative if and only if the tangent part of its bitension field vanishes.

Until now a special attention was paid to biconservative surfaces in space forms. Thus, when the ambient space form is $3$-dimensional, such surfaces were completely classified in \cite{CMOP} and \cite{Fu} and then biconservative surfaces with constant mean curvature in $4$-dimensional space forms were described in \cite{MOR2}. 

We will extend this study to surfaces with parallel mean curvature vector field ({\it PMC surfaces}) and, more generaly, to those having constant mean curvature ({\it CMC surfaces}) in product spaces of type $M^n(c)\times\mathbb{R}$, where $M^n(c)$ is a space form. While all PMC surfaces in space forms are biconservative, we will see that in this new setting the situation is quite different.

Another notion that we deal with in our paper is that of submanifolds with {\it finite total curvature}, i.e., those submanifolds $\Sigma^m$ in a Riemannian manifold $\bar M$ that satisfy 
$$
\int_{\Sigma^m}|\phi|^m\ dv<+\infty,
$$
where $\phi$ is the traceless part of the second fundamental form of $\Sigma^m$. One of the most interesting results concerning these submanifolds is that obtained by P.~B\'erard, M.~do Carmo, and W.~Santos in the very first paper to treat this subject \cite{bds}, where they proved that a CMC surface with $|H|>\sqrt{-c}$ and finite total curvature in a space form $M^3(c)$, $c\leq 0$, must be compact. This theorem was then extended to PMC submanifolds in a space form $M^n(c)$, $c\leq 0$, in \cite{ccs}. 

We will show how, in the case of CMC biconservative surfaces, these results hold in any ambient space whose sectional curvature is negative and bounded. 

In our paper, we prove a classification result for non-minimal PMC biconservative surfaces in $\mathbb{S}^n\times\mathbb{R}$ and $\mathbb{H}^n\times\mathbb{R}$ and, moreover, when these surfaces are not vertical cylinders nor they lie entirely in $\mathbb{S}^n$ or $\mathbb{H}^n$, we also find their explicit (local) equation (Theorem \ref{thm:main1}). While, as we will see from this theorem, such surfaces do not exist when $n=3$, we find examples of CMC biconservative (but not PMC) surfaces in $M^3(c)\times\mathbb{R}$ that do not lie in $M^3(c)$ and are not vertical cylinders (Theorem \ref{thm:biCMC}). Moreover, we study the biharmonicity of these examples (Theorem \ref{thm:CMCbiharmonic}).  

Next, we study CMC biconservative surfaces in Hadamard manifolds as a generalization of the study of CMC biconservative
surfaces in $M^n(c)\times\mathbb{R}$ with $c<0$. First, we show that CMC biconservative surfaces in a Riemannian manifold satisfy a Simons type inequality and then, as CMC surfaces in a Hadamard manifold also satisfy a Sobolev type inequality, we use these two results to prove that if a non-minimal CMC biconservative surface in a Hadamard manifold with bounded sectional curvature has finite total curvature $\int_{\Sigma^2}|\phi_H|^2\ dv<+\infty$ and the norm of its second fundamental form is bounded, then the function $|\phi_H|$ goes to $0$ uniformly at infinity, where $\phi_H$ is the traceless part of the shape operator  of the surface in the direction of its mean curvature vector field $H$ (Theorem \ref{zero:PMC}). This theorem allows us to prove a result on the compactness of some of these surfaces (Theorem \ref{theo-compact0}). We note that we use a more general notion of finite total curvature than the original one in \cite{bds}. 

{\bf Acknowledgments.} The first author would like to thank the Department of Mathematics of the Federal University of Bahia in Salvador for providing a very stimulative work environment during the preparation of this paper.

\section{Preliminaries}

A natural generalization of \textit{harmonic maps}, i.e., those maps $\psi:(M,g)\rightarrow(\bar M,h)$ between two Riemannian manifolds that are
critical points of the \textit{energy functional}
$$
E(\psi)=\frac{1}{2}\int_{M}|d\psi|^{2}\ dv,
$$
are the {\it biharmonic maps}, i.e., the critical points of the
\textit{bienergy functional}
$$
E_{2}(\psi)=\frac{1}{2}\int_{M}|\tau(\psi)|^{2}\ dv,
$$
where $\tau(\psi)=\trace\nabla d\psi$ is the tension field that
vanishes for harmonic maps. This generalization was first
suggested by J.~Eells and J.~H.~Sampson in \cite{JEJS}.

The Euler-Lagrange equation for the
bienergy functional, derived in \cite{GYJ}, is
$$
\tau_{2}(\psi)=\Delta\tau(\psi)-\trace\bar R(d\psi,\tau(\psi))d\psi=0,
$$
where $\tau_{2}(\psi)$ is the \textit{bitension field} of $\psi$, $\Delta=\trace(\nabla^{\psi})^2 =\trace(\nabla^{\psi}\nabla^{\psi}-\nabla^{\psi}_{\nabla})$ is the rough Laplacian defined on
sections of $\psi^{-1}(T\bar M)$ and
$\bar R$ is the curvature tensor of $\bar M$, given by $\bar R(X,Y)Z=[\bar\nabla_X,\bar\nabla_Y]Z-\bar\nabla_{[X,Y]}Z$. 

The {\it stress-energy tensor} associated to a variational problem, described in \cite{H} by D.~Hilbert, is a symmetric $2$-covariant tensor $S$ conservative at critical points, i.e., satisfying $\Div S=0$. 

Such a tensor $S$, given by $S=(1/2)|d\psi|^2g-\psi^{\ast}h$, was employed in the study of harmonic maps by P.~Baird and J.~Eells in \cite{BE} and A.~Sanini in \cite{S}. It satisfies $\Div S=-\langle\tau(\psi),d\psi\rangle$ and, therefore, $\Div S$ vanishes when $\psi$ is harmonic. Since for isometric immersions $\tau(\psi)$ is normal, it follows that $\Div S=0$ is always satisfied in this case.

The stress-energy tensor $S_2$ of the bienergy, first studied in \cite{J} and then in \cite{CMOP,Fu,LMO,MOR,MOR2}, is given by
\begin{align*}
S_2(X,Y)=&\frac{1}{2}|\tau(\psi)|^2\langle X,Y\rangle+\langle d\psi,\nabla\tau(\psi)\rangle\langle X,Y\rangle\\&-\langle d\psi(X),\nabla_Y\tau(\psi)\rangle-\langle d\psi(Y),\nabla_X\tau(\psi)\rangle
\end{align*}
and it satisfies
$$
\Div S_2=-\langle\tau_2(\psi),d\psi\rangle.
$$
If $\psi:(\Sigma^m,g)\rightarrow(\bar M,h)$ is an isometric immersion, then we have $\Div S_2=-\tau_2(\psi)^{\top}$ and thus $\Div S_2$ does not always vanish.

\begin{definition} A submanifold $\psi:\Sigma^m\rightarrow\bar M$ of a Riemannian manifold $\bar M$ is called a {\it biconservative submanifold} if $\Div S_2=0$, i.e., $\tau_2(\psi)^{\top}=0$.
\end{definition}

The following decomposition of the bitension field was obtained in \cite{BMO} (see also \cite{O}).

\begin{theorem}[\cite{BMO}]\label{thm:split} A submanifold $\psi:\Sigma^m\rightarrow\bar M$ in a Riemannian manifold $\bar M$, with second fundamental form $\sigma$,
mean curvature vector field $H$, and shape operator $A$, is biharmonic if and only if the normal and the tangent components of $\tau_2(\psi)$ vanish, i.e., respectively
$$
-\Delta^{\perp}H+\trace\sigma(\cdot,A_H\cdot)+\trace(\bar R(\cdot,H)\cdot)^{\perp}=0
$$
and
$$
\frac{m}{2}\grad|H|^2+2\trace A_{\nabla^{\perp}_{\cdot}H}(\cdot)+2\trace(\bar R(\cdot,H)\cdot)^{\top}=0,
$$
where $\Delta^{\perp}$ is the Laplacian in the normal bundle and $\bar R$ is the curvature tensor of $\bar M$.
\end{theorem}

\begin{corollary}\label{bicons} A submanifold $\Sigma^m$ in a Riemannian manifold $\bar M$ is biconservative if and only if
$$
\frac{m}{2}\grad|H|^2+2\trace A_{\nabla^{\perp}_{\cdot}H}(\cdot)+2\trace(\bar R(\cdot,H)\cdot)^{\top}=0.
$$
\end{corollary}

We also recall the following theorem that will be used later on.

\begin{theorem}[\cite{LO}]\label{thm:Q} Let $\Sigma^2$ be a biconservative oriented surface in a Riemannian manifold $\bar M$. Then the $(2,0)$-part of the Hopf quadratic form $Q$, defined on $\Sigma^2$ by
$$
Q(X,Y)=\langle\sigma(X,Y),H\rangle,
$$
is holomorphic if and only if the mean curvature $|H|$ of the surface is constant.
\end{theorem}

Now, let us consider $\Sigma^2$ an isometrically immersed surface in a Riemannian manifold $\bar M$. The second fundamental form $\sigma$ of $\Sigma^2$ is defined by the equation of Gauss
$$
\bar\nabla_XY=\nabla_XY+\sigma(X,Y),
$$
for any tangent vector fields $X$ and $Y$, where $\bar\nabla$ and $\nabla$ are the Levi-Civita connections on $\bar M$ and $\Sigma^2$, respectively, and we locally identified $d\psi(\nabla_XY)$ with $\nabla_XY$. Then the mean curvature vector field $H$ of $\Sigma^2$ is given by $H=(1/2)\trace\sigma$. The shape operator $A$ and the normal connection $\nabla^{\perp}$ are defined by the equation of Weingarten
$$
\bar\nabla_XV=-A_VX+\nabla^{\perp}_XV,
$$
for any tangent vector field $X$ and any normal vector field $V$. 

\begin{definition} If the mean curvature vector field $H$ of a surface $\Sigma^2$ is
parallel in the normal bundle, i.e.,
$\nabla^{\perp} H=0$, then $\Sigma^2$ is called a \textit{PMC surface}. 
\end{definition}
In space of constant curvature, a PMC submanifold trivially is biconservative. It would be then interesting to study PMC biconservative submanifolds in spaces whose sectional curvature is not constant.

Next, let $M^n(c)$ be a space form, i.e., a simply-connected $n$-dimensional manifold with
constant sectional curvature $c$, and consider the product manifold $\bar M=M^n(c)\times\mathbb{R}$. Then, the curvature tensor $\bar R$ of $\bar M$ is given by
\begin{align}\label{eq:barR}
\bar R(X,Y)Z=&c\{\langle Y, Z\rangle X-\langle X, Z\rangle Y-\langle Y,\xi\rangle\langle Z,\xi\rangle X+\langle X,\xi\rangle\langle Z,\xi\rangle Y\\\nonumber &+\langle X,Z\rangle\langle Y,\xi\rangle\xi-\langle Y,Z\rangle\langle X,\xi\rangle\xi\},
\end{align}
where $\xi$ is the unit vector field tangent to $\mathbb{R}$.

\begin{definition} A surface $\Sigma^2$ in $M^n(c)\times\mathbb{R}$ is called a \textit{vertical cylinder} over $\gamma$ if $\Sigma^2=\pi^{-1}(\gamma)$, where $\pi:M^n(c)\times\mathbb{R}\rightarrow M^n(c)$ is the projection map and $\gamma:I\subset\mathbb{R}\rightarrow M^n(c)$ is a curve in $M^n(c)$.
\end{definition}

It is easy to see that vertical cylinders $\Sigma^2=\pi^{-1}(\gamma)$ are characterized by the fact that $\xi$ is tangent to $\Sigma^2$.

We end this section recalling the following definition of Frenet curves that we will use later. 

\begin{definition} Let $\gamma:I\subset\mathbb{R}\rightarrow \bar M^{n+1}$ be a curve parametrized by
arc-length. Then $\gamma$ is called a {\it Frenet curve of
osculating order} $r$, $1\leq r\leq n+1$, if there exist $r$
orthonormal vector fields $\{X_1=\gamma',\ldots,X_r\}$ along
$\gamma$ such that
$$
\bar\nabla_{X_{1}}X_{1}=\kappa_{1}X_{2},\quad
\bar\nabla_{X_{1}}X_{i}=-\kappa_{i-1}X_{i-1} + \kappa_{i}X_{i+1},\quad\ldots\quad,
\bar\nabla_{X_{1}}X_{r}=-\kappa_{r-1}X_{r-1},
$$
for all $i\in\{2,\ldots,r-1\}$, where $\{\kappa_{1},\kappa_{2},\ldots,\kappa_{r-1}\}$ are positive
functions on $I$ called the {\it curvatures} of $\gamma$. A Frenet curve of osculating order $r$ is called a {\it helix of
order $r$} if $\kappa_i=\cst>0$ for $1\leq i\leq r-1$. A helix
of order $2$ is called a {\it circle}, and a helix of order $3$ is
simply called {\it helix}.
\end{definition}

\section{PMC biconservative surfaces in $M^n(c)\times\mathbb{R}$}

Let $\Sigma^2$ be a non-minimal PMC surface in $\bar M=M^n(c)\times\mathbb{R}$. For the sake of simplicity we will consider only the cases $c=\pm 1$, i.e., $M^n(c)$ is either $\mathbb{S}^n$ or $\mathbb{H}^n$. It follows, from Corollary \ref{bicons}, that $\Sigma^2$ is biconservative if and only if 
$$
(\trace\bar R(\cdot,H)\cdot))^{\top}=0,
$$
where $H$ is the mean curvature vector field of our surface, which, using \eqref{eq:barR}, is equivalent to
\begin{equation}\label{eq:condbicons}
c\langle H,N\rangle T=0,
\end{equation}
where $T$ and $N$ are the tangent and the normal components of $\xi$, respectively.

As our result is of local nature, in the following, we will split our study as $|T|=0$, or $|T|=1$, or $|T|\in(0,1)$ on $\Sigma^2$.

\textbf{Case I.} Let us assume that $|T|=0$ at any point of $\Sigma^2$. This means that $\xi$ is orthogonal to our surface or, equivalently, that $\Sigma^2$ lies in $M^n(c)$. Obviously, equation \eqref{eq:condbicons} holds automatically in this case and $\Sigma^2$ is biconservative. Moreover, $\Sigma^2$ is a PMC surface in a space form and these surfaces were classified in \cite{Y}.

\textbf{Case II.} If $|T|=1$ on the surface, then $\xi$ is tangent to $\Sigma^2$ at any point, which means that $\Sigma^2$ is a vertical cylinder over a circle with curvature $\kappa=2|H|$ in $M^2(c)$ (see \cite{AdCT}). Moreover, $H$ is orthogonal to $\xi$ and then \eqref{eq:condbicons} implies that $\Sigma^2$ is biconservative in this case too.

\textbf{Case III.} Henceforth we shall assume that $|T|\in(0,1)$ at any point of the surface $\Sigma^2$. Also assume that $\Sigma^2$ is biconservative and orientable. We will see that, in this case, our surface has no pseudo-umbilical points.

First, from Theorem \ref{thm:Q}, it follows that either $H$ is umbilical everywhere and then $\Sigma^2$ lies in $M^n(c)$ (or equivalently $|T|=0$), which is a contradiction, or $H$ is not umbilical on the surface, which implies that $\Sigma^2$ lies in $M^4\times\mathbb{R}$ (see \cite{AdCT}). 

Next, since $|T|\neq 0$ on $\Sigma^2$, from \eqref{eq:condbicons}, we have that $H$ is orthogonal to $\xi$, that implies
$$
X(\langle H,\xi\rangle)=0,
$$
or, equivalently, as $\nabla^{\perp}H=0$ and $\bar\nabla\xi=0$,
$$
\langle A_HT,X\rangle=0,
$$
for any vector field $X$ tangent to the surface, so $A_HT=0$.

Now, let us consider the global, positive oriented orthonormal frame field $\{E_1=T/|T|,E_2\}$ on the surface and, since $A_HE_1=0$, we note that this frame field diagonalizes $A_H$. From the equation of Ricci 
$$
\langle R^{\perp}(X,Y)U,V\rangle=\langle [A_U,A_V]X,Y\rangle+\langle \bar R(X,Y)U,V\rangle,
$$
where $X$, $Y$ are tangent vector fields and $U$, $V$ are normal vector fields, we see, using the expression \eqref{eq:barR} of the curvature tensor $\bar R$ and the fact that $H$ is parallel, that $A_H$ and $A_U$ commute for any vector field $U$ normal to $\Sigma^2$, which shows that $\{E_1,E_2\}$ diagonalizes the second fundamental form $\sigma$ of our surface.

Next, consider the following decomposition of $\xi$
\begin{equation}\label{eq:xi}
\xi=\cos\theta E_1+\sin\theta E_3,
\end{equation}
where $\theta\in(0,\pi/2)$ is a local angle function, and $\{E_3=N/|N|,E_4,E_5=H/|H|\}$ is a global orthonormal frame field in the normal bundle.

First, we will prove the following lemma.

\begin{lemma}\label{lemma:calcul} The following equations hold on the surface $\Sigma^2$$:$
\begin{enumerate}
\item $\nabla E_1=\nabla E_2=0$$;$

\item $\theta=\cst$$;$

\item $\nabla^{\perp}_{E_1}E_3=-\cot\theta\sigma(E_1,E_1)$$;$

\item $\nabla^{\perp}_{E_2}E_3=0$$;$

\item $A_3=A_{E_3}=0$$;$

\item $\nabla^{\perp}_{E_1}(\sigma(E_2,E_2))=-c\cos\theta\sin\theta E_3$$;$

\item $\nabla^{\perp}_{E_2}(\sigma(E_2,E_2))=0$$.$
\end{enumerate}
Moreover, we have $c=1$, i.e., $\Sigma^2$ lies in $\mathbb{S}^4\times\mathbb{R}$, and $|\sigma(E_1,E_1)|=\sin\theta$$.$
\end{lemma}

\begin{proof} From \eqref{eq:xi}, since $\bar\nabla_{E_1}\xi=0$, we get
$$
-E_1(\theta)\sin\theta E_1+\cos\theta\bar\nabla_{E_1}E_1+E_1(\theta)\cos\theta E_3+\sin\theta\bar\nabla_{E_1}E_3=0
$$
and then, since $\{E_1,E_2\}$ diagonalizes $\sigma$,
\begin{equation}\label{eq:1}
\nabla_{E_1}E_1=\nabla_{E_1}E_2=0,
\end{equation}
\begin{equation}\label{eq:2}
A_3E_1=A_{E_3}E_1=-E_1(\theta)E_1
\end{equation}
and
\begin{equation}\label{eq:3}
\nabla^{\perp}_{E_1}E_3=-\cot\theta(\sigma(E_1,E_1)+E_1(\theta)E_3).
\end{equation}

In the same way, from $\bar\nabla_{E_2}\xi=0$, one obtains $\nabla_{E_2}^{\perp}E_3=0$,
\begin{equation}\label{eq:4}
E_2(\theta)=0,
\end{equation}
and
\begin{equation}\label{eq:5}
\cos\theta\nabla_{E_2}E_1-\sin\theta A_3E_2=0.
\end{equation}

Next, we will compute $\nabla^{\perp}_{E_2}(\sigma(E_2,E_2))$ and $\nabla^{\perp}_{E_1}(\sigma(E_2,E_2))$. Using the Codazzi equation of $\Sigma^2$ in $\bar M$
\begin{equation}\label{eq:Codazzi}
(\bar R(X,Y)Z)^{\perp}=(\nabla^{\perp}_X\sigma)(Y,Z)-(\nabla^{\perp}_Y\sigma)(X,Z),
\end{equation}
where $X$, $Y$, $Z$ are tangent vector fields, the expression \eqref{eq:barR} of the curvature tensor $\bar R$, and equations \eqref{eq:1}, since $\nabla^{\perp}H=0$ and $\{E_1,E_2\}$ diagonalizes $\sigma$, we have
\begin{align*}
\nabla^{\perp}_{E_2}(\sigma(E_2,E_2))&=-\nabla^{\perp}_{E_2}(\sigma(E_1,E_1))=-((\nabla^{\perp}_{E_2}\sigma)(E_1,E_1)+2\sigma(E_1,\nabla_{E_2}E_1))\\
&=-((\nabla^{\perp}_{E_1}\sigma)(E_1,E_2)+2\sigma(E_1,\nabla_{E_2}E_1))+(\bar R(E_2,E_1)E_1)^{\perp}\\&=-2\sigma(E_1,\nabla_{E_2}E_1))=-2\langle\nabla_{E_2}E_1,E_2)\sigma(E_1,E_2)\\&=0
\end{align*}
and
\begin{align}\label{eq:6}
\nabla^{\perp}_{E_1}(\sigma(E_2,E_2))&=(\nabla^{\perp}_{E_1}\sigma)(E_2,E_2)-2\sigma(E_2,\nabla_{E_1}E_2)\\\nonumber
&=(\nabla^{\perp}_{E_2}\sigma)(E_1,E_2)+(\bar R(E_1,E_2)E_2)^{\perp}\\\nonumber &=-\sigma(E_2,\nabla_{E_2}E_1)-\sigma(E_1,\nabla_{E_2}E_2)-c\cos\theta\sin\theta E_3.
\end{align}

Now, since $A_HE_1=0$, we know that $\langle\sigma(E_2,E_2),H\rangle=2|H|^2$ and then 
$$
E_1(\langle\sigma(E_2,E_2),H\rangle)=0.
$$
The facts that $H$ is parallel and orthogonal to $E_3$, using \eqref{eq:6}, lead to
$$
\langle\sigma(E_2,\nabla_{E_2}E_1)+\sigma(E_1,\nabla_{E_2}E_2),H\rangle=0,
$$
or, equivalently, since $A_HE_1=0$,
$$
\langle\nabla_{E_2}E_1,E_2\rangle|H|^2=0,
$$
which means that
\begin{equation}\label{eq:7}
\nabla_{E_2}E_1=\nabla_{E_2}E_1=0.
\end{equation}

From equations \eqref{eq:1} and \eqref{eq:7}, we see that $\nabla E_1=\nabla E_2=0$. Moreover, since $\trace A_3=0$, from \eqref{eq:2} and \eqref{eq:5}, we have
$A_3E_2=E_1(\theta)E_2=0$ and then $A_3=0$ and $E_1(\theta)=0$. From \eqref{eq:4}, it follows that the function $\theta$ is constant.

Finally, the shape operator $A$ of the surface is given, with respect to $\{E_1,E_2\}$, by
$$
A_3=A_{\frac{N}{|N|}}=0,\quad A_4=\left(\begin{array}{cc}\lambda&0\\0&-\lambda\end{array}\right),\quad A_5=A_{\frac{H}{|H|}}=\left(\begin{array}{cc}0&0\\0&2|H|\end{array}\right),
$$
where $\lambda$ is a smooth function on $\Sigma^2$, and we have $\sigma(E_1,E_1)=\lambda E_4$. Then, from the Gauss equation of $\Sigma^2$ in $\bar M$
\begin{align}\label{eq:Gauss}
\langle R(X,Y)Z,W\rangle=&\langle\bar R(X,Y)Z,W\rangle+\langle\sigma(Y,Z),\sigma(X,W)\rangle\\\nonumber &-\langle\sigma(X,Z),\sigma(Y,W)\rangle,
\end{align}
where $X$, $Y$, $Z$, $W$ are tangent vector fields and $R$ is the curvature tensor of the surface, we obtain the Gaussian curvature $K$ of $\Sigma^2$ as
$$
K=c\sin^2\theta-\lambda^2.
$$

Since the equations $\nabla E_1=\nabla E_2=0$ imply that $\Sigma^2$ is flat, it follows that $\lambda^2=c\sin^2\theta$, which means that $c>0$ and, therefore, $c=1$, which completes the proof.
\end{proof}

\begin{remark} We note that if $\Sigma^2$ is a PMC biconservative surface in $M^4(c)\times\mathbb{R}$ with $|T|\in (0,1)$, then it lies in a totally geodesic submanifold $M^3(c)\times\mathbb{R}$ if and only if $A_4=0$. But, from the Gauss equation of the surface, we get that there are no PMC biconservative
surfaces in $M^3(c)\times\mathbb{R}$ with $|T|\in(0,1)$.
\end{remark}

Next, we consider the immersion of $\mathbb{S}^4\times\mathbb{R}$ in $\mathbb{R}^5\times\mathbb{R}$ and denote by $\widetilde\nabla$ the Levi-Civita connection on $\mathbb{R}^5\times\mathbb{R}$. Then the integral curves of $E_1$ and $E_2$, thought as curves in $\mathbb{R}^5\times\mathbb{R}$, are characterized by the following two lemmas.

\begin{lemma}\label{lemma:inte1} The integral curves $\delta$ of $E_1$ are helices in $\mathbb{R}^5\times\mathbb{R}$ with curvatures
$$
\kappa_1=\sin\theta\sqrt{1+\sin^2\theta}\quad\textnormal{and}\quad\kappa_2=\cos\theta\sqrt{1+\sin^2\theta},
$$
where $\theta=\cst\in(0,\pi/2)$.
\end{lemma}

\begin{proof} First, since $\nabla E_1=0$, we have
$$
\widetilde\nabla_{E_1}E_1=\sigma(E_1,E_1)-\langle E_1-\langle E_1,\xi\rangle\xi,E_1-\langle E_1,\xi\rangle\xi\rangle\eta=\sigma(E_1,E_1)-\sin^2\theta\eta,
$$
where $\eta$ is the unit vector field orthogonal to $\mathbb{S}^4$ in $\mathbb{R}^5$.

From the first Frenet equation $\widetilde\nabla_{E_1}E_1=\kappa_1X_2$ of the curve $\delta$, where $\{X_1=E_1,X_2,\ldots,X_r\}$ is the Frenet frame field along $\delta$, we get, using Lemma \ref{lemma:calcul},
\begin{equation}\label{eq:k1}
\kappa_1^2=\sin^2\theta(1+\sin^2\theta)=\cst,
\end{equation}
and then
\begin{equation}\label{eq:Frenet2}
\widetilde\nabla_{E_1}X_2=\frac{1}{\kappa_1}(\widetilde\nabla_{E_1}(\sigma(E_1,E_1))-\sin^2\theta\widetilde\nabla_{E_1}\eta).
\end{equation}

From Lemma \ref{lemma:calcul}, since $H$ is parallel and $\{E_1,E_2\}$ diagonalizes $\sigma$, one obtains
\begin{align}\label{eq:Frenet21}
\widetilde\nabla_{E_1}(\sigma(E_1,E_1))&=-A_{\sigma(E_1,E_1)}E_1-\nabla^{\perp}_{E_1}(\sigma(E_2,E_2))\\\nonumber &=-\sin^2\theta E_1+\cos\theta\sin\theta E_3.
\end{align}

We also have
\begin{equation}\label{eq:Frenet22}
\widetilde\nabla_{E_1}\eta=E_1-\langle E_1,\xi\rangle\xi=\sin^2\theta E_1-\cos\theta\sin\theta E_3.
\end{equation}

Replacing \eqref{eq:Frenet21} and \eqref{eq:Frenet22} in \eqref{eq:Frenet2}, we get
$$
\widetilde\nabla_{E_1}X_2=\frac{1}{\kappa_1}(1+\sin^2\theta)(-\sin^2\theta E_1+\cos\theta\sin\theta E_3)
$$
and, from the second Frenet equation $\widetilde\nabla_{E_1}X_2=-\kappa_1E_1+\kappa_2X_3$ of $\delta$ and \eqref{eq:k1}, it follows that $X_3=E_3$ and
$$
\kappa_2=\frac{\cos\theta\sin\theta(1+\sin^2\theta)}{\kappa_1}=\cos\theta\sqrt{1+\sin^2\theta}=\cst.
$$

Again using Lemma \ref{lemma:calcul} and the expressions of $\kappa_1$ and $\kappa_2$, we have
\begin{align*}
\widetilde\nabla_{E_1}E_3&=-A_3E_1+\nabla^{\perp}_{E_1}E_3-\langle E_1-\langle E_1,\xi\rangle\xi,E_3-\langle E_3,\xi\rangle\xi\rangle\eta\\&=-\cot\theta(\sigma(E_1,E_1)-\sin^2\theta\eta)\\&=-\kappa_2X_2,
\end{align*}
which means that $\delta$ is a helix, and we conclude.
\end{proof}

\begin{remark}\label{rem:1} In the proof of Lemma \ref{lemma:inte1} we have seen that, when $c=1$,
$$
\widetilde\nabla_{E_1}\eta=\sin^2\theta E_1-\cos\theta\sin\theta E_3\quad\textnormal{and}\quad\widetilde\nabla_{E_1}(\sigma(E_1,E_1)=-\sin^2\theta E_1+\cos\theta\sin\theta E_3.
$$
From the latter equation, we obtain
$$
\widetilde\nabla_{E_1}(\sigma(E_2,E_2))=2\widetilde\nabla_{E_1}H-\widetilde\nabla_{E_1}(\sigma(E_1,E_1))=\sin^2\theta E_1-\cos\theta\sin\theta E_3,
$$ 
since $\widetilde\nabla_{E_1}H=-A_HE_1+\nabla^{\perp}_{E_1}H-\langle E_1-\langle E_1,\xi\rangle\xi,H\rangle\eta=0$, and, therefore, 
$$
\widetilde\nabla_{E_1}\widetilde\nabla_{E_2}E_2=\widetilde\nabla_{E_1}(\sigma(E_2,E_2)-\eta)=0.
$$
\end{remark}

\begin{lemma}\label{lemma:inte2} The integral curves $\gamma$ of $E_2$ are plane circles in $\mathbb{R}^5\times\mathbb{R}$ with curvature $\kappa=\sqrt{1+4|H|^2+\sin^2\theta}$, where $\theta=\cst\in(0,\pi/2)$.
\end{lemma}

\begin{proof} First, we have $\widetilde\nabla_{E_2}E_2=\sigma(E_2,E_2)-\eta$ and then, from the first Frenet equation $\widetilde\nabla_{E_2}E_2=\kappa X_2$ of $\gamma$, where $\{X_1=E_2,X_2,\ldots, X_r\}$ is the Frenet frame field along $\gamma$, we get 
$$
\kappa^2=1+|\sigma(E_2,E_2)|^2=1+4|H|^2+\sin^2\theta=\cst,
$$
since, by Lemma \ref{lemma:calcul}, we know that $|\sigma(E_1,E_1)|=\sin\theta$ and $\sigma(E_1,E_1)$ is orthogonal to $H$.

Next, using Lemma \ref{lemma:calcul}, we obtain the second Frenet equation of $\gamma$ 
\begin{align*}
\widetilde\nabla_{E_2}X_2&=\frac{1}{\kappa}(\widetilde\nabla_{E_2}(\sigma(E_2,E_2))-\widetilde\nabla_{E_2}\eta)=\frac{1}{\kappa}(-A_{\sigma(E_2,E_2)}E_2-E_2)\\\nonumber &=\frac{1}{\kappa}(-2A_HE_2+A_{\sigma(E_1,E_1)}E_2-E_2)\\\nonumber &=-\kappa E_2,
\end{align*}
that shows that our curve is a circle.
\end{proof}

Now, we can state the main result of this section.

\begin{theorem}\label{thm:main1} Let $\Sigma^2$ be a PMC biconservative surface with mean curvature vector field $H$ in $\bar M=M^n(c)\times\mathbb{R}$, $c=\pm 1$ and $H\neq 0$. Then either
\begin{enumerate}
\item $\Sigma^2$ either is a minimal surface of an umbilical hypersurface of $M^n(c)$ or it is a CMC surface in a
$3$-dimensional umbilical submanifold of $M^n(c)$$;$ or

\item $\Sigma^2$ is a vertical cylinder over a circle in $M^2(c)$ with curvature $\kappa=2|H|$$;$ or

\item $\Sigma^2$ lies in $\mathbb{S}^4\times\mathbb{R}\subset\mathbb{R}^5\times\mathbb{R}$ and, as a surface in $\mathbb{R}^5\times\mathbb{R}$, is locally given by 
\begin{align*}
X(u,v)=&\frac{1}{a}\{C_3+\sin\theta(D_1\cos(au)+D_2\sin(au))\}+(u\cos\theta+b)\xi\\&+\frac{1}{\kappa}(C_1(\cos v-1)+C_2\sin v),
\end{align*}
where $\theta\in(0,\pi/2)$ is a constant, $a=\sqrt{1+\sin^2\theta}$, $b$ is a real constant, $\kappa=\sqrt{1+4|H|^2+\sin^2\theta}$, $C_1$ and $C_2$ are two constant orthonormal vectors in $\mathbb{R}^5\times\mathbb{R}$ such that $C_1\perp\xi$ and $C_2\perp\xi$, $C_3$ is a unit constant vector such that $\langle C_3,C_1\rangle=a/\kappa\in(0,1)$, $C_3\perp C_2$, and $C_3\perp\xi$, and $D_1$ and $D_2$ are two constant orthonormal vectors in the orthogonal complement of $\Span\{C_1,C_2,C_3,\xi\}$ in $\mathbb{R}^5\times\mathbb{R}$.
\end{enumerate}
\end{theorem}

\begin{proof} We only have to study the case when the surface $\Sigma^2$ is not pseudo-umbilical and $|T|\in(0,1)$. In order to do that, we will use the same method employed in \cite{CMOP} to study biconservative surfaces in space forms. 

We consider again the local orthonormal frame field $\{E_1=T/|T|,E_2\}$ and let $\gamma$ be an integral curve of $E_2$ parametrized by arc-length. Then, from Lemma \ref{lemma:inte2}, we know that $\gamma$ is a circle with curvature $\kappa=\sqrt{1+4|H|^2+\sin^2\theta}$ in $\mathbb{R}^5\times\mathbb{R}$ and, therefore, it can be written as
\begin{equation}\label{eq:gamma}
\gamma(s)=c_0+c_1\cos(\kappa s)+c_2\sin(\kappa s),\quad c_0,c_1,c_2\in\mathbb{R}^5\times\mathbb{R},
\end{equation}
where $|c_1|=|c_2|=1/\kappa$ and $\langle c_1,c_2\rangle=0$.

At an arbitrary point $p_0\in\Sigma^2$ we consider $\delta(u)$ an integral curve of $E_1$, with $\delta(0)=p_0$, and the flow $\phi$ of $E_1$ near $p_0$. We note that $\delta(u)$ is a helix characterized in Lemma \ref{lemma:inte1}. Now, for all $u\in(-\omega,\omega)$ and $s\in(-\epsilon,\epsilon)$, we have
$$
\phi_{\delta(u)}(s)=c_0(u)+c_1(u)\cos(\kappa s)+c_2(u)\sin(\kappa s),
$$
with
$$
\delta(u)=c_0(u)+c_1(u),\quad |c_1(u)|=|c_2(u)|=\frac{1}{\kappa},\quad\langle c_1(u),c_2(u)\rangle=0,
$$
and, therefore, the surface can be parametrized locally by
$$
X(u,s)=\phi_{\delta(u)}(s).
$$

Next, $X(u,s)$ can be reparametrized using $u$ and $v=\kappa s$ as the new parameters, with $u\in(-\omega,\omega)$ and $v\in(-\kappa\epsilon,\kappa\epsilon)$, and we have
$$
X(u,v)=c_0(u)+\frac{1}{\kappa}(C_1(u)\cos v+C_2(u)\sin v),
$$
where $C_1(u)=\kappa c_1(u)$ and $C_2(u)=\kappa c_2(u)$.

Since at $v=0$ the integral curves of $E_2$ start from $\delta$, we have 
$$
\delta(u)=X(u,0)=c_0(u)+\frac{1}{\kappa}C_1(u)
$$
and then
\begin{equation}\label{eq:X}
X(u,v)=\delta(u)+\frac{1}{\kappa}(C_1(u)(\cos v-1)+C_2(u)\sin v).
\end{equation}

From \eqref{eq:gamma} it follows that $C_2=\kappa c_2=\gamma'(0)=E_2(\gamma(0))$, that is 
$$
C_2(u)=E_2(\delta(u)),
$$
and also $-\kappa^2c_1=\gamma''(0)=(\widetilde\nabla_{E_2}E_2)(\gamma(0))$, which gives 
$$
C_1(u)=\kappa c_1(u)=-\frac{1}{\kappa}(\widetilde\nabla_{E_2}E_2)(\delta(u)).
$$ 

Now, using Lemma \ref{lemma:calcul} and Remark \ref{rem:1}, we have 
$$
\frac{dC_1}{du}=-\frac{1}{\kappa}\widetilde\nabla_{E_1}\widetilde\nabla_{E_2}E_2=0\quad\textnormal{and}\quad\frac{dC_2}{du}=\widetilde\nabla_{E_1}E_2=0,
$$
which means that $C_1$ and $C_2$ are constant orthonormal vectors and that the image of parametrization \eqref{eq:X} is given by a $1$-parameter family of circles centered in $\delta(u)-(1/\kappa)C_1$ and passing through the points of $\delta(u)$ lying in planes parallel to the one spanned by $C_1$ and $C_2$. Moreover, from Lemma \ref{lemma:calcul}, one also obtains that $C_1\perp\xi$ and $C_2\perp\xi$. 

Next, we will determine the explicit equation of $\delta(u)$. In order to do that, let us consider the vector field 
$$
C(u)=\delta''(u)+(1+\sin^2\theta)\eta(\delta(u))
$$ 
along $\delta(u)$. It is then easy to verify, using Remark \ref{rem:1}, that $C'(u)=0$, which means that $C(u)=C$ is a constant vector. From Lemmas \ref{lemma:calcul}, \ref{lemma:inte1}, and \ref{lemma:inte2}, we also get that $\langle C,C_1\rangle=a^2/\kappa$, where $a=\sqrt{1+\sin^2\theta}$, $C\perp C_2$, $C\perp\xi$, and $|C|=a$. Moreover, $C$, $C_1$, and $C_2$ are linearly independent.

Next, consider $\delta_1(u)=\delta(u)-\langle\delta(u),\xi\rangle\xi$. Since $\widetilde\nabla_{E_1}\xi=0$, it follows that 
$$
\delta_1'(u)=E_1-\cos\theta\xi
$$ 
and then $|\delta_1'(u)|=\sin\theta$. Differentiating $\delta_1'(u)$ along $\delta(u)$, since $\widetilde\nabla_{E_1}E_1=\delta''(u)=C-a^2\eta(\delta(u))=C-a^2\delta_1(u)$, we can see that $\delta_1(u)$ satisfies
$$
\delta_1''(u)+a^2\delta_1(u)=C,
$$
that shows that
$$
\delta_1(u)=\frac{1}{a^2}C+\frac{1}{a}(F_1\cos(au)+F_2\sin(au)),
$$
where $F_1$ and $F_2$ are two constant vectors in $\mathbb{R}^5\times\mathbb{R}$. Since $\delta_1'(u)$ is orthogonal to $\xi$, we have that $F_1\perp\xi$ and $F_2\perp\xi$. Also, from $|\delta_1'(u)|=\sin\theta$, one obtains that $F_1\perp F_2$ and $|F_1|=|F_2|=\sin\theta$. Then, considering $C_3=(1/a)C$, $D_1=(1/\sin\theta)F_1$ and $D_2=(1/\sin\theta)F_2$, we can write
$$
\delta_1(u)=\frac{1}{a}(C_3+\sin\theta(D_1\cos(au)+D_2\sin(au))),
$$
where $C_3$, $D_1$ and $D_2$ are unit constant vectors such that $D_1\perp D_2$. It follows that
$$
\delta_1'(u)=-D_1\sin(au)+D_2\cos(au),
$$
which, taking into account that $\delta_1'=E_1-\cos\theta\xi$ is orthogonal to $C_1$, $C_2$, and $C_3$, implies that $D_1$ and $D_2$ are vectors in the orthogonal complement of $\Span\{C_1,C_2,C_3\}$ in $\mathbb{R}^5\times\mathbb{R}$. 

Finally, since $(d/du)(\langle\delta(u),\xi\rangle)=\langle E_1,\xi\rangle=\cos\theta$ along $\delta(u)$, we have $\langle\delta(u),\xi\rangle=u\cos\theta+b$, where $b$ is a real constant. Hence, we conclude that $\delta(u)$ is given by
$$
\delta(u)=\frac{1}{a}\{C_3+\sin\theta(D_1\cos(au)+D_2\sin(au))\}+(u\cos\theta+b)\xi,
$$
which completes the proof.
\end{proof}

\begin{remark} We note that surfaces given by the third case of Theorem \ref{thm:main1} lie in the Riemannian product of a small hypersphere of $\mathbb{S}^4$ with $\mathbb{R}$. In order to see this, let us consider $X_1(u,v)=X(u,v)-\langle X(u,v),\xi\rangle\xi$ and the constant vector $\widetilde C$ in $\mathbb{R}^5$, orthogonal to $\xi$, given by $\widetilde{C}=(1/a)C_3-(1/\kappa)C_1$. Then, is easy to verify that $\langle X_1(u,v)-\widetilde C,\widetilde C\rangle=0$ and that $X_1(u,v)$ lies in $\mathbb{S}^4$, which shows that $X_1(u,v)$ actually lies in $\mathbb{S}^4\cap\pi$, where $\pi$ is a hyperplane of $\mathbb{R}^5$ that passes through $\widetilde C$ such that $\widetilde C$ is orthogonal to $\pi$. Moreover, since $|X_1(u,v)-\widetilde C|^2=(a^2+\kappa^2\sin^2\theta)/a^2\kappa^2$, we get that $X(u,v)$ lies in $\mathbb{S}^3(\widetilde C,\sqrt{a^2+\kappa^2\sin^2\theta}/a\kappa)\times\mathbb{R}$, where $\mathbb{S}^3(\widetilde C,\sqrt{a^2+\kappa^2\sin^2\theta}/a\kappa)$ is the $3$-dimensional sphere in the hyperplane $\pi$, centered in $\widetilde C$ and with radius $\sqrt{a^2+\kappa^2\sin^2\theta}/a\kappa$.
\end{remark}

\begin{remark} Lemma \ref{lemma:calcul} implies that the angle between a PMC biconservative surface given by the third case of Theorem \ref{thm:main1} and $\xi$ is constant. PMC surfaces with this property in spaces of type $M^n(c)\times\mathbb{R}$ were classified in \cite{DF}. However, here we use a different method that, as we have seen, allows us to find the explicit equation of PMC biconservative surfaces in the third case of the theorem.
\end{remark} 

\begin{remark} By similar arguments to those used in Lemma \ref{lemma:calcul}, it can be proved that PMC biconservative surfaces in $M^4(c)\times\mathbb{R}$, with $c\neq 0$ an arbitrary constant, that are not pseudo-umbilical nor vertical cylinders exist only when $c>0$. The local equations of such surfaces in $\mathbb{R}^5\times\mathbb{R}$ can be obtained working in the same way as in Theorem \ref{thm:main1}.
\end{remark}

\begin{remark} We note that, since all non-minimal PMC biharmonic surfaces in $\mathbb{S}^n\times\mathbb{R}$ that do not lie in $\mathbb{S}^n$ are vertical cylinders (see \cite[Theorem~5.6]{FOR}), the surfaces described in the third case of Theorem \ref{thm:main1} are not biharmonic.
\end{remark}

From Theorem \ref{thm:main1}, we know that the mean curvature vector field $H$ of a PMC biconservative surface in $M^n(c)\times\mathbb{R}$ is orthogonal to $\xi$. Let us now consider a CMC biconservative surface $\Sigma^2$ in $M^n(c)\times\mathbb{R}$ with $H$ orthogonal to $\xi$. As we will show in the next section, in general, $\Sigma^2$ is a not a PMC surface. The following result, however, highlights a particular case when these conditions imply that $H$ is parallel. 

\begin{proposition} Let $\Sigma^2$ be a genus zero CMC biconservative surface in $M^n(c)\times\mathbb{R}$ with mean curvature vector field $H$ orthogonal to $\xi$. Then $\Sigma^2$ is pseudo-umbilical and it lies in $M^n(c)$. Moreover, when $n=4$, $\Sigma^2$ is a PMC surface.
\end{proposition}

\begin{proof} Let us consider local isothermal coordinates $(U;x,y)$. Then we have $ds^2=\lambda^2(dx^2+dy^2)$ for some positive function
$\lambda$ on $U$ and $\{\partial/\partial x,\partial/\partial y\}$ is positively oriented. We will denote
$$
z=x+{\rm i}y,\quad \partial_z=\frac{\partial}{\partial z}=\frac{1}{2}\Big(\frac{\partial}{\partial x}-{\rm i}\frac{\partial}{\partial y}\Big),\quad \partial_{\bar z}=\frac{\partial}{\partial \bar z}=\frac{1}{2}\Big(\frac{\partial}{\partial x}+{\rm i}\frac{\partial}{\partial y}\Big).
$$

From Theorem \ref{thm:Q} we have that the $(2,0)$-part of the quadratic form $Q$ is holomorphic, that implies that $Q(\partial z,\partial z)$ vanishes, since the genus of $\Sigma^2$ is zero. Hence, $\Sigma^2$ is pseudo-umbilical.

Next, we define the quadratic form $\mathcal Q$ on our surface by 
$$
\mathcal Q(x,y)=\langle X,\xi\rangle\langle Y,\xi\rangle.
$$

It is easy to verify that $\bar\nabla_{\partial\bar z}\partial z=(1/2)\lambda^2 H$ and then we get that
$$
\partial\bar z(\mathcal Q(\partial z,\partial z))=\lambda^2\langle\partial z,\xi\rangle\langle H,\xi\rangle=0,
$$
i.e., the $(2,0)$-part of $\mathcal Q$ is holomorphic and, therefore, vanishes. We have just proved that $\partial z$ is orthogonal to $\xi$, which means that $\Sigma^2$ lies in $M^n(c)$.

Finally, when $n=4$, since $\Sigma^2$ is pseudo-umbilical, we use a result in \cite{C} to conclude.
\end{proof}

\section{CMC biconservatitve surfaces in $M^3(c)\times\mathbb{R}$}

As we have seen in Theorem \ref{thm:main1}, there are no PMC biconservative surfaces in $M^3(c)\times\mathbb{R}$, where $c=\pm 1$, that do not lie in $M^3(c)$ nor are
vertical cylinders. It is then interesting to see if there is possible to find examples of CMC biconservative surfaces $\Sigma^2$ with $|T|\in(0,1)$ in these spaces. 

We first note that it can be easily verified that a CMC surface in $M^3(c)\times\mathbb{R}$ with $|T|=0$ is biconservative since it actually lies in $M^3(c)$, and also that a CMC surface with $|T|=1$ and constant mean curvature $|H|$ in $M^3(c)\times\mathbb{R}$ is a vertical cylinder over a curve in $M^3(c)$ with constant first curvature $\kappa_1=2|H|$ and, therefore, a biconservative surface. In both cases, the mean curvature vector field is orthogonal to $\xi$. 

When $|T|\in(0,1)$ we have the following characterization of CMC biconservative surfaces in $M^3(c)\times\mathbb{R}$ whose mean curvature vector field $H$ is orthogonal to $\xi$.

\begin{theorem}\label{thm:biCMC} Let $\Sigma^2$ be a CMC biconservative surface in $M^3(c)\times\mathbb{R}$, $c\neq 0$, with mean curvature vector field $H\neq 0$ orthogonal to $\xi$ and $|T|\in(0,1)$. Then $\Sigma^2$ is flat and it is locally given by $X=X(u,v)$, where $X:D\subset\mathbb{R}^2\rightarrow M^3(c)\times\mathbb{R}$ is an isometric immersion, $D$ is an open set in $\mathbb{R}^2$, and either
\begin{enumerate}
\item $\Sigma^2$ is pseudo-umbilical, $c<0$, $|H|^2=-c(1-|T|^2)$, the integral curve of $X_u$ is a helix such that $\langle X_u,\xi\rangle=|T|$, with curvatures $\kappa_1^1=|H|$ and $\kappa_2^1=\sqrt{-c}|T|$, and the integral curve of $X_v$ is a circle such that $\langle X_v,\xi\rangle=0$, with curvature $\kappa_1^2=|H|$$;$ or

\item $|H|^2>-c(1-|T|^2)$ and the integral curves of $X_u$ and $X_v$ are helices in $M^3(c)\times\mathbb{R}$ satisfying
$$
\langle X_u,\xi\rangle=a\quad\textnormal{and}\quad\langle X_v,\xi\rangle=b,
$$
where $a$ and $b$ are two real constants such that 
$$
0<a^2+b^2=|T|^2<1\quad\textnormal{and}\quad|H|^2+c(1-a^2-b^2)>0,
$$ 
and with curvatures
$$
\kappa_1^1=|H|+\sqrt{|H|^2+c(1-a^2-b^2)},\quad\kappa_2^1=\frac{|a|}{\sqrt{1-a^2-b^2}}\kappa_1^1
$$ 
and
$$
\kappa_1^2=\Big||H|-\sqrt{|H|^2+c(1-a^2-b^2)}\Big|,\quad\kappa_2^2=\frac{|b|}{\sqrt{1-a^2-b^2}}\kappa_1^2,
$$
respectively.
\end{enumerate}
\end{theorem}

\begin{proof} Since our surface is biconservative, it follows, from Theorem \ref{thm:main1}, that it cannot have parallel mean curvature vector field. Therefore, $\nabla^{\perp}H\neq 0$, which means that there exists an open subset $U\subset\Sigma^2$ such that $\nabla^{\perp}H\neq 0$ at any point $p\in U$. Let us now consider a local orthonormal frame field $\{E_1,E_2\}$ on $U$ and an orthonormal frame field $\{E_3=H/|H|,E_4=N/|N|\}$ in the normal bundle of $\Sigma^2$. Then $\{E_1,E_2,E_3,E_4\}$ can be extended to local orthonormal frame field on an open subset of $M^3(c)\times\mathbb{R}$. Denote by $\omega^A_B$ the corresponding connection $1$-forms on this subset given by
$$
\bar\nabla_XE_A=\omega_A^B(X)E_B.
$$
Then, from Corollary \ref{bicons}, we get that the biconservative equation becomes
$$
\trace A_{\nabla_{\cdot}^{\perp}H}(\cdot)=|H|(\omega_3^4(E_1)A_4E_1+\omega_3^4(E_2)A_4E_2)=0,
$$
that is equivalent to
\begin{equation}\label{eq:system}
\begin{cases}
\omega_3^4(E_1)\langle A_4E_1,E_1\rangle+\omega_3^4(E_2)\langle A_4E_2,E_1\rangle=0\\
\omega_3^4(E_1)\langle A_4E_1,E_2\rangle+\omega_3^4(E_2)\langle A_4E_2,E_2\rangle=0.
\end{cases}
\end{equation}
Since $\nabla^\perp H\neq 0$, we have $(\omega_3^4(E_1))^2+(\omega_3^4(E_2))^2\neq 0$, which, using \eqref{eq:system}, implies that 
$$
\langle A_4E_1,E_1\rangle\langle A_4E_2,E_2\rangle-\langle A_4E_2,E_1\rangle\langle A_4E_1,E_2\rangle=0.
$$

Next, the fact that $E_4$ is orthogonal to $H$ shows that $\trace A_4=0$ and then, since $A_4$ is symmetric, one obtains
$$
|A_4|^2=-2\langle A_4E_1,E_1\rangle\langle A_4E_2,E_2\rangle+2\langle A_4E_2,E_1\rangle\langle A_4E_1,E_2\rangle=0,
$$
i.e., $A_4=0$. We note that, since $\bar\nabla\xi=0$ and $N=|N|E_4$, we also have $\nabla T=A_N=0$ and $|T|,|N|\in(0,1)$ are constants.

We know, from Theorem \ref{thm:Q}, that $H$ either is umbilical at any point of $\Sigma^2$, and then $\Sigma^2$ is pseudo-umbilical, or $H$ is umbilical only on a closed set without interior points. In the second case, $H$ is not umbilical on an open dense connected set $W$. 

Let us first treat the case when our surface is pseudo-umbilical. Then $\{E_1,E_2\}$ diagonalizes $A_3$ and, moreover, since $|T|\in(0,1)$, we can choose $E_1=T/|T|$, which implies that $\nabla E_1=\nabla E_2=0$. 

Using the Codazzi equation \eqref{eq:Codazzi} of $\Sigma^2$ in $M^3(c)\times\mathbb{R}$, first with $X=E_1$, $Y=Z=E_2$ and then with $X=Z=E_1$, $Y=E_2$, and taking the inner product with $E_4$, one obtains
\begin{equation}\label{eq:omega1}
\omega_3^4(E_1)=-\frac{c|T||N|}{|H|}=\cst\quad\textnormal{and}\quad\omega_3^4(E_2)=0.
\end{equation}

Next, since $\langle E_3,\xi\rangle=0$, we have $\langle\bar\nabla_{E_1}E_3,\xi\rangle=0$, which gives
\begin{equation}\label{eq:omega2}
\omega_3^4(E_1)=\frac{|T||H|}{|N|}.
\end{equation}

From \eqref{eq:omega1} and \eqref{eq:omega2} one sees that $|H|^2=-c|N|^2$, that implies $c<0$ and, using the Gauss equation \eqref{eq:Gauss}, that the surface is flat. Moreover, one obtains 
\begin{equation}\label{eq:omega3}
\omega_3^4(E_1)=\pm\sqrt{-c}|T|. 
\end{equation}

As we have seen, we have $\nabla E_1=\nabla E_2=0$ and then $[E_1,E_2]=0$, which means that there exists a local parametrization $X=X(u,v)$ of $\Sigma^2$ such that $X_u=E_1$ and $X_v=E_2$. 

In the following, we shall determine the curvatures of the integral curves $\gamma_1$ and $\gamma_2$ of $X_u$ and $X_v$, respectively. 

From the first Frenet equation $\bar\nabla_{E_1}E_1=\kappa_1^1X^1_2$ of $\gamma_1$, since $\bar\nabla_{E_1}E_1=\sigma(E_1,E_1)=|H|E_3$, it follows that the first curvature of $\gamma_1$ is $\kappa_1^1=|H|$ and $X^1_2=E_3$. The second Frenet equation $\bar\nabla_{E_1}X_2^1=-\kappa_1^1E_1+\kappa_2^1X^1_3$, together with 
$$
\bar\nabla_{E_1}X^1_2=\bar\nabla_{E_1}E_3=-A_3E_1+\nabla^{\perp}_{E_1}E_3=-|H|E_1+\omega_3^4(E_1)E_4
$$
and \eqref{eq:omega3}, leads to $\kappa_2^1=|\omega_3^4(E_1)|=\sqrt{-c}|T|$ and $X^1_3=(\omega_3^4(E_1)/|\omega_3^4(E_1)|)E_4$. Finally, the third Frenet equation of $\gamma_1$ is 
$$
\bar\nabla_{E_1}X^1_3=\frac{\omega_3^4(E_1)}{|\omega_3^4(E_1)|}\bar\nabla_{E_1}E_4=-|\omega_3^4(E_1)|E_3=-\kappa_2^1X^1_2.
$$

The first Frenet equation $\bar\nabla_{E_2}E_2=\kappa_1^2X^2_2$ of $\gamma_2$ and $\bar\nabla_{E_2}E_2=\sigma(E_2,E_2)=|H|E_3$ give $\kappa_1^2=|H|$ and $X_2^2=E_3$. Then, the second Frenet equation $\bar\nabla_{E_2}X_2^2=-\kappa_1^2E_2+\kappa_2^2X^2_3$ and \eqref{eq:omega1} imply that $\kappa_2^2=0$, which shows that $\gamma_2$ is a circle.

Let us now consider the case when $\Sigma^2$ is not pseudo-umbilical.

First, we choose $E_1$ and $E_2$ such that $A_3E_i=\lambda_i E_i$, $i=1,2$, and $\lambda_1>\lambda_2$. Since $A_4=0$, we have
$$
\sigma(E_1,E_1)=\lambda_1E_3,\quad\sigma(E_1,E_2)=0,\quad\sigma(E_2,E_2)=\lambda_2E_3
$$
and also
$$
\begin{cases}
\nabla_{E_1}E_1=\omega_1^2(E_1)E_2,\quad\nabla_{E_1}E_2=-\omega_1^2(E_1)E_1\\
\nabla_{E_2}E_1=\omega_1^2(E_2)E_1,\quad\nabla_{E_2}E_2=-\omega_1^2(E_2)E_1.
\end{cases}
$$

Next, we again use the Codazzi equation \eqref{eq:Codazzi} with $X=Z=E_1$ and $Y=E_2$, to obtain, taking the inner product first with $E_3$ and then with $E_4$, 
\begin{equation}\label{eq:Codazzi1}
E_2(\lambda_1)=(\lambda_1-\lambda_2)\omega_1^2(E_1)
\end{equation}
and
\begin{equation}\label{eq:Codazzi2}
\lambda_1\omega_3^4(E_2)+c\langle T,E_2\rangle|N|=0.
\end{equation}
In the same way, this time taking $X=E_1$ and $Y=Z=E_2$, we get
\begin{equation}\label{eq:Codazzi3}
E_1(\lambda_2)=(\lambda_1-\lambda_2)\omega_1^2(E_2)
\end{equation}
and
\begin{equation}\label{eq:Codazzi4}
\lambda_2\omega_3^4(E_1)+c\langle T,E_1\rangle|N|=0.
\end{equation}

From $\langle E_3,\xi\rangle=0$ we have $\langle\bar\nabla_{E_1}E_3,\xi\rangle=0$ and $\langle\bar\nabla_{E_2}E_3,\xi\rangle=0$, that are
\begin{equation}\label{eq:C5}
\lambda_1\langle T,E_1\rangle-\omega_3^4(E_1)|N|=0
\end{equation}
and
\begin{equation}\label{eq:C6}
\lambda_2\langle T,E_2\rangle-\omega_3^4(E_2)|N|=0.
\end{equation}

Now, from \eqref{eq:Codazzi2}, \eqref{eq:Codazzi4}, \eqref{eq:C5}, and \eqref{eq:C6} it follows 
$$
\lambda_1\lambda_2+c|N|^2=0,
$$
which, using the Gauss equation \eqref{eq:Gauss}, shows that $\Sigma^2$ is flat. Moreover, since $\lambda_1+\lambda_2=2|H|$, we get
\begin{equation}\label{eq:lambda}
\lambda_i=|H|\pm\sqrt{|H|^2+c|N|^2}=\cst,\quad i=1,2.
\end{equation}

From \eqref{eq:Codazzi2}, \eqref{eq:Codazzi4}, and \eqref{eq:lambda} one sees that 
$$
\begin{cases}
\omega_3^4(E_1)=-\dfrac{ca\sqrt{1-a^2-b^2}}{|H|-\sqrt{|H|^2+c(1-a^2-b^2)}}\\\omega_3^4(E_2)=-\dfrac{cb\sqrt{1-a^2-b^2}}{|H|+\sqrt{|H|^2+c(1-a^2-b^2)}},
\end{cases}
$$
where $a=\langle T,E_1\rangle$ and $b=\langle T,E_2\rangle$.

The fact that $\lambda_1$ and $\lambda_2$ are constants, together with \eqref{eq:Codazzi1} and \eqref{eq:Codazzi3}, leads to $\omega_1^2(E_1)=\omega_1^2(E_2)=0$, i.e., $\nabla E_1=\nabla E_2=0$. Since $\nabla T=0$, we can also see that $\omega_3^4(E_i)$, $i=1,2$, are constants and then that the Ricci equation does not provide any other supplementary information about $\Sigma^2$.

Finally, since $[E_1,E_2]=0$, there exists a local parametrization $X=X(u,v)$ of $\Sigma^2$ such that $X_u=E_1$ and $X_v=E_2$. We conclude by computing the curvatures of the integral curves of $X_u$ and $X_v$ in the same way as in the case when the surface is pseudo-umbilical.
\end{proof}

\begin{theorem}\label{thm:CMCbiharmonic} If $\Sigma^2$ is a CMC biharmonic surface in $M^3(c)\times\mathbb{R}$, $c\neq 0$, with mean curvature vector field $H\neq 0$ orthogonal to $\xi$ and $|T|\in(0,1)$, then $c>0$, $b^2>a^2$, and $\Sigma^2$ is one of the non-pseudo-umbilical CMC biconservative surfaces in Theorem \ref{thm:biCMC}, with
\begin{equation}\label{eq:Hc}
|H|^2=\frac{c(1-a^2-b^2)(b^2-a^2)^2}{4(1-a^2)(1-b^2)}.
\end{equation}
\end{theorem}

\begin{proof} The normal part of the bitension field $\tau_2$ of surfaces in Theorem \ref{thm:biCMC}, whose general expression is given by Theorem~\ref{thm:split}, is
$$
\tau_2^{\perp}=2(2|H|^2-2c)H,
$$
in the pseudo-umbilical case, and 
\begin{align*}
\tau_2^{\perp}=&2\Big\{4|H|^2-c(a^2+b^2)+\frac{1}{1-a^2-b^2}\Big(a^2\big(|H|+\sqrt{|H|^2+c(1-a^2-b^2)}\big)^2\\ &+b^2\big(|H|-\sqrt{|H|^2+c(1-a^2-b^2)}\big)^2\Big)\Big\}H,
\end{align*}
when the surface is not pseudo-umbilical. 

When $c<0$, it is easy to see that $\tau_2^{\perp}$ does not vanish, which means that our surfaces are not biharmonic in this case.

When $c>0$, we have that $\tau_2^{\perp}=0$ is equivalent to 
$$
|H|(2-a^2-b^2)+\sqrt{|H|^2+c(1-a^2-b^2)}(a^2-b^2)=0,
$$
from where it follows that $b^2>a^2$ and the mean curvature of the surface is given by equation \eqref{eq:Hc}.
\end{proof}

\section{CMC biconservative surfaces in Hadamard manifolds}

In order to prove some compactness results for CMC biconservative surfaces in $M^n(c)\times\mathbb{R}$, with $c<0$, we will work in a more general setting where the ambient space is a Hadamard manifold, i.e., a Riemannian manifold that is complete simply-connected and has non-positive sectional curvature everywhere. 

We will begin by showing that a CMC biconservative surface in a Riemannian manifold satisfies a Simons type equation.

\begin{theorem}\label{thm:delta} 
Let $\Sigma^2$ be a non-minimal CMC biconservative surface in a Riemannian manifold $\bar M$ with mean curvature vector field $H$ and shape operator $A$. 
Then
$$
\frac{1}{2}\Delta|\phi_H|^2=2K|\phi_H|^2+|\nabla\phi_H|^2,
$$
where $\phi_H=A_H-|H|^2\id$ is the traceless part of $A_H$ and $K$ is the Gaussian curvature of the surface.
\end{theorem}

\begin{proof} We first recall the following Simons type formula (equation $2.8$ in \cite{CY}). Let $\Sigma^m$ be an $m$-dimensional Riemannian manifold and consider a symmetric operator $S$ on $\Sigma^m$ that satisfies the Codazzi equation $(\nabla_XS)Y=(\nabla_YS)X$. Then, we have
\begin{equation}\label{delta}
\frac{1}{2}\Delta|S|^2=|\nabla S|^2+\sum_{i=1}^{m}\lambda_i(\trace S)_{ii}+\frac{1}{2}\sum_{i,j=1}^{m}R_{ijij}(\lambda_i-\lambda_j)^2,
\end{equation}
where $\lambda_i$, $1\leq i\leq m$, are the eigenvalues of $S$, and $R_{ijkl}$ are the components of the Riemannian curvature of $\Sigma^m$.

In our case, where $\Sigma^2$ is a biconservative surface, using isothermal coordinates $(x,y)$ on the surface, we get, by a straightforward computation,
$$
\Big(\nabla_{\frac{\partial}{\partial x}}A_H\Big)\frac{\partial}{\partial y}-\Big(\nabla_{\frac{\partial}{\partial y}}A_H\Big)\frac{\partial}{\partial x}=\frac{3}{2}\Big(-\frac{\partial}{\partial y}(|H|^2)\frac{\partial}{\partial x}+\frac{\partial}{\partial x}(|H|^2)\frac{\partial}{\partial y}\Big),
$$
which, since $\Sigma^2$ has constant mean curvature, shows that $A_H$, and then $\phi_H$, satisfies the Codazzi equation. Since $\phi_H$ is symmetric and traceless, we conclude using equation \eqref{delta} with $S=\phi_H$.
\end{proof}

\begin{corollary}\label{cor:simons} Let $\Sigma^2$ be a CMC biconservative surface in a Riemannian manifold $\bar M$ and assume that $\Sigma^2$ is compact and $K\geq 0$. Then $\nabla A_H=0$ and the surface is pseudo-umbilical or flat.
\end{corollary}

\begin{proof} From Theorem \ref{thm:delta} we get that $\int_{\Sigma^2}(2K|\phi_H|^2+|\nabla\phi_H|^2)\ dv=0$
and, since $2K|\phi_H|^2+|\nabla\phi_H|^2\geq 0$, one obtains that $\nabla A_H=\nabla\phi_H=0$ and, at any point on the surface, $K=0$ or $\phi_H=0$. We conclude using Theorem \ref{thm:Q}, that shows that $\phi_H$ either vanishes at any point of $\Sigma^2$ or only on a closed set without interior points. Hence, if the surface is not pseudo-umbilical, it follows that $K=0$ on an open dense set in $\Sigma^2$, and then the Gaussian curvature vanishes everywhere.
\end{proof}

\begin{corollary}\label{simons2} 
Let $\Sigma^2$ be a non-minimal CMC biconservative surface in a Riemannian manifold $\bar M$, with sectional curvature bounded from below by a constant $K_0$, such that 
$\mu=\sup_{\Sigma^2}(|\sigma|^2-(1/|H|)^2|A_{H}|^2)<+\infty$. Then
$$
-\Delta|\phi_H|\leq a|\phi_H|^3+b|\phi_H|,
$$
where $a$ and $b$ are constants depending on $K_0$, $|H|$, and $\mu$.
\end{corollary}

\begin{proof} 
Let $\{E_3=H/H,E_4,\ldots,E_{n}\}$ be a local orthonormal frame field in 
the normal bundle, where $n$ is the dimension of the ambient space $\bar M$, and denote $A_{\alpha}=A_{E_{\alpha}}$. Then, from the Gauss equation  \eqref{eq:Gauss} of $\Sigma^2$ in $\bar M$, we obtain the following expression of the Gaussian curvature of $\Sigma^2$
\begin{align*}
K&=\langle\bar R(E_1,E_2)E_2,E_1\rangle+\sum_{\alpha=3}^n\det A_{\alpha}\\&=\langle\bar R(E_1,E_2)E_2,E_1\rangle+|H|^2-\frac{1}{2|H|^2}|\phi_H|^2-\frac{1}{2}(|\sigma|^2-|A_3|^2),
\end{align*}
where $\{E_1,E_2\}$ is a local orthonormal frame field on the surface. Since by hypothesis we have $\langle\bar R(E_1,E_2)E_2,E_1\rangle\geq K_0$, we get that
\begin{equation}\label{eq:K0}
K\geq K_0+|H|^2-\frac{1}{2|H|^2}|\phi_H|^2-\frac{\mu}{2},
\end{equation}
and then, from Theorem \ref{thm:delta}, one obtains
\begin{align*}
\frac{1}{2}\Delta|\phi_H|^2\geq 2\left(K_0+|H|^2-\frac{1}{2|H|^2}|\phi_H|^2-\frac{\mu}{2}\right)|\phi_H|^2+|\nabla\phi_H|^2.
\end{align*}

Since $|\nabla|\phi_H||\leq|\nabla \phi_H|$, we easily get that
\begin{align*}
-\Delta|\phi_H|&\leq \frac{1}{|H|^2}|\phi_H|^3-\left(2K_0+2|H|^2-\mu\right)|\phi_H|,
\end{align*}
which completes the proof.
\end{proof}

Now, let us consider a CMC surface $\Sigma^2$ in a Hadamard manifold. We recall that such a surface satisfies a Sobolev inequality of the form
\begin{equation}\label{sobolev}
\forall f\in C_0^{\infty}(\Sigma),\quad ||f||_2\leq A||\nabla f||_1+B||f||_1,
\end{equation}
where $||f||_p=(\int_{\Sigma}|f|^p \ dv)^{1/p}$ is the $L^p$-norm of the function $f$ and $A$ and $B$ are constants that depends only on the mean curvature $|H|$ of the surface (see \cite{hs}).

Next, let us fix a point $x_0\in\Sigma^2$ on the surface and consider the Riemannian distance function $d(x_0,x)$ to $x_0$ and the following open domains  
$$
B(R)=\{x\in\Sigma | d(x_0,x)<R\}\quad\textnormal{and}\quad E(R)=\{x\in\Sigma | d(x_0,x)>R\}.
$$

Now we can state the following theorem.

\begin{theorem}\label{zero:PMC}
Let $\Sigma^2$ be a complete non-minimal CMC biconservative surface in a Hadamard manifold $\bar M$, with sectional curvature bounded from below by a constant $K_0<0$, such that the norm of its second fundamental form $\sigma$ is bounded and 
\begin{equation}\label{eq:finite-curvature}
\int_{\Sigma^2}|\phi_H|^2 \ dv < +\infty.
\end{equation}
Then the function $u=|\phi_H|$ goes to zero uniformly at infinity. More exactly, there exist positive constants $C_0$ and $C_1$, depending on $K_0$, $|H|$, and $\mu=\sup_{\Sigma^2}(|\sigma|^2-(1/|H|)^2|A_{H}|^2)$, and a positive radius $R_{\Sigma^2}$, determined by $C_1\displaystyle\int_{E(R_{\Sigma^2})}u^2 \ dv\leq 1$, such that
$$
||u||_{\infty, E(2R)}\leq C_0\int_{\Sigma^2}u^2 \ dv,
$$
for all $R\geq  R_{\Sigma^2}$. Moreover, there exist some positive constants $D_0$ and $E_0$, depending on $K_0$, $|H|$, and $\mu$, such that the inequality
$\displaystyle\int_{\Sigma^2}u^2 \ dv \leq D_0$ implies
$$
||u||_{\infty} \leq E_0 \int_{\Sigma^2}u^2 \ dv.
$$
\end{theorem}

\begin{proof} Since the function $u=|\phi_H|$ satisfies the Sobolev inequality \eqref{sobolev} and the Simons type inequality in Corollary \ref{simons2}, we work as in the proof of \cite[Theorem 4.1]{bds} and come to the conclusion.
\end{proof}

We note that, when $n=2$, we have $\mu=0$ and then it is easy to see that \eqref{eq:finite-curvature} implies that $|\sigma|$ is bounded. Therefore, we have the following corollary.

\begin{corollary}\label{zero}
Let $\Sigma^2$ be a complete non-minimal CMC biconservative surface in a $3$-dimensional Hadamard manifold $\bar M$, with sectional curvature bounded from below by a constant $K_0<0$, such that
$$
\int_{\Sigma^2}|\phi_H|^2 \ dv < +\infty.
$$
Then the function $u=|\phi_H|$ goes to zero uniformly at infinity. More exactly, there exist positive constants $C_0$ and $C_1$, depending on $K_0$ and $|H|$, and a positive radius $R_{\Sigma^2}$, determined by $C_1\displaystyle\int_{E(R_{\Sigma^2})}u^2 \ dv\leq 1$, such that
$$
||u||_{\infty, E(2R)}\leq C_0\int_{\Sigma^2}u^2 \ dv,
$$
for all $R\geq  R_{\Sigma^2}$. Moreover, there exist some positive constants $D_0$ and $E_0$, depending on $K_0$ and $|H|$, such that the inequality
$\displaystyle\int_{\Sigma^2}u^2 \ dv \leq D_0$ implies
$$
||u||_{\infty} \leq E_0 \int_{\Sigma^2}u^2 \ dv.
$$
\end{corollary}

In the following we will use Theorem \ref{zero:PMC} to prove a compactness result for CMC biconservative surfaces in Hadamard manifolds.

\begin{theorem}\label{theo-compact0}
Let $\Sigma^2$ be a complete non-minimal CMC biconservative surface in a Hadamard manifold $\bar M$, with sectional curvature bounded from below by a constant $K_0<0$, such that the norm of its second fundamental form $\sigma$ is bounded,
$$
\int_{\Sigma^2}|\phi_H|^2\ dv < +\infty,
$$
and $|H|^2>(\mu-2K_0)/2$, where $\mu=\sup_{\Sigma^2}(|\sigma|^2-(1/|H|^2)|A_{H}|^2)$. Then $\Sigma^2$ is compact.
\end{theorem}

\begin{proof} Using inequality \eqref{eq:K0} and Theorem \ref{zero:PMC}, we have that the superior limit at infinity of the Gaussian curvature $K$ of $\Sigma^2$ is positive. It follows that the negative part $K^-$ of $K$ has compact support and, therefore, 
$$
\int_{\Sigma^2} |K^-|\ dv <+\infty,
$$
which implies, using \cite[Theorem 1]{w}, that also the positive part $K^+$ of $K$ satisfies 
$$
\int_{\Sigma^2}K^+\ dv<+\infty.
$$

Next, since outside a compact set $\Omega$ we have $K^+\geq k/2>0$, where
$$
k=K_0+|H|^2-\frac{\mu}{2},
$$
it follows that $\vol(\Sigma\backslash\Omega)<+\infty$. Since the volume of a complete non-compact surface is infinite (see \cite{f}), we conclude that $\Sigma^2$ is compact.
\end{proof}

When $n=2$, we use Theorem \ref{zero:PMC} and Corollary \ref{zero} to prove our next result.

\begin{corollary}\label{theo-compact1}
Let $\Sigma^2$ be a complete non-minimal CMC biconservative surface in a $3$-dimensional Hadamard manifold $\bar M$, with sectional curvature bounded from below by a constant $K_0<0$, such that 
$$
\int_{\Sigma^2}|\phi_H|^2\ dv < +\infty,
$$
and $|H|^2>-K_0$. Then $\Sigma^2$ is compact.
\end{corollary}


\begin{thebibliography}{99}

\bibitem{AdCT} H.~Alencar, M.~do Carmo, and R.~Tribuzy, \textit{A
Hopf Theorem for ambient spaces of dimensions higher than three},
J. Differential Geometry 84(2010), 1--17.

\bibitem{BE} P. Baird, J. Eells, {\it A conservation law for harmonic maps}, Geometry Symposium, Utrecht 1980, 1--25, Lecture Notes in Math. 894, Springer, Berlin-New York, 1981. 

\bibitem{BMO} A.~Balmu\c s, S.~Montaldo, and C.~Oniciuc, \textit{Biharmonic PNMC submanifolds in spheres}, Ark. Mat. 51(2013), 197--221. 

\bibitem{bds} P.~B\'erard, M.~do Carmo, and W. Santos, {\it Complete hypersurfaces with constant mean curvature and finite total curvature}, Ann. Global Anal. Geom. 16(1998), 273--290.

\bibitem{CMOP} R. Caddeo, S. Montaldo, C. Oniciuc, and P. Piu, {\it Surfaces in three-dimensional space forms with divergence-free stress-bienergy tensor}, Ann. Mat. Pura Appl. (4) 193(2014), 529--550.

\bibitem{ccs} M.~do Carmo, L.-F.~ Cheung, and W. Santos, {\it On the compactness of constant mean curvature hypersurfaces with finite total curvature}, Arch. Math. (Basel) 73(1999), 216--222. 

\bibitem{C} B.-Y.~Chen, {\it Minimal hypersurfaces of an $m$-sphere}, Proc. Amer. Math. Soc. 29(1971), 375--380.

\bibitem{CY} S.-Y. Cheng, S.-T. Yau, \textit{Hypersurfaces with constant scalar curvature}, Math. Ann. 225(1977), 195--204.

\bibitem{JEJS} J.~Eells, J. H.~Sampson, \textit{Harmonic mappings of Riemannian manifolds}, Amer. J. Math. 86(1964), 109--160.

\bibitem{DF} D.~Fetcu, {\it A classification result for helix surfaces with parallel mean curvature in product spaces}, \texttt{arXiv:1312.3196}, preprint.

\bibitem{FOR} D.~Fetcu, C.~Oniciuc, and H.~Rosenberg, {\it Biharmonic submanifolds with parallel mean curvature in $\mathbb{S}^n\times\mathbb{R}$}, J. Geom. Anal. 23(2013), 2158--2176.

\bibitem{f}  K.~R.~Frensel, {\it Stable complete surfaces with constant mean curvature}, Bol. Soc. Bras. Mat. 27(1996), 129--144.

\bibitem{Fu} Y.~Fu, {\it Explicit classification of biconservative surfaces in Lorentz $3$-space forms}, Ann. Mat. Pura Appl., to appear.

\bibitem{H} D.~Hilbert, {\it Die grundlagen der physik}, Math. Ann. 92(1924), 1--32.

\bibitem{hs} D.~Hoffman, J.~Spruck, {\it Sobolev and isoperimetric inequalities for Riemannian submanifolds}, Comm. Pure. Appl. Math. 27(1974), 715--727. 

\bibitem{GYJ} G. Y.~Jiang, \textit{$2$-harmonic maps and
their first and second variational formulas}, Chinese Ann. Math.
Ser. A7(4)(1986), 389--402.

\bibitem{J} G. Y.~Jiang, \textit{The conservation law for $2$-harmonic maps between Riemannian manifolds}, Acta Math. Sinica 30(1987), 220--225.

\bibitem{LMO} E. Loubeau, S. Montaldo, and C. Oniciuc, {\it The stress-energy tensor for biharmonic maps}, Math. Z. 259(2008), 503--524.

\bibitem{LO} E. Loubeau, C. Oniciuc, {\it Biharmonic CMC surfaces in spheres}, Pacific J. Math., to appear.

\bibitem{MOR} S. Montaldo, C. Oniciuc, and A. Ratto, {\it Proper Biconservative immersions into the Euclidean space}, \texttt{arXiv:1312.3053}, preprint. 

\bibitem{MOR2} S. Montaldo, C. Oniciuc, and A. Ratto, {\it Biconservative surfaces}, \texttt{arXiv:1406.6774}, preprint.

\bibitem{O} Y.-L. Ou, {\it Biharmonic hypersurfaces in Riemannian manifolds}, Pacific J. Math. 248(2010), 217--232.

\bibitem{S} A. Sanini, {\it Applicazioni tra variet\`a riemanniane con energia critica rispetto a deformazioni di metriche}, Rend. Mat. 3(1983), 53--63.

\bibitem{w}  B.~White, {\it Complete surfaces of finite total curvature}, J. Differential Geom. 26(1987), 315--326.

\bibitem{Y} S.-T.~Yau, \textit{Submanifolds with constant mean curvature. I}, Amer. J. Math. 96(1974), 346--366.

\end{thebibliography}
\end{document}